\documentclass[preprint,12pt]{elsarticle}

\usepackage{amsmath,amssymb,amsthm}

\usepackage{algorithm}
\usepackage{algpseudocode}
\usepackage{mathtools}
\usepackage{geometry}

\usepackage{graphicx}
\usepackage{subcaption}
\usepackage{booktabs}
\usepackage{makecell}

\usepackage{enumitem}
\usepackage{url}
\usepackage{float}

\newtheorem{theorem}{Theorem}
\newtheorem{lemma}{Lemma}
\newtheorem{corollary}{Corollary}

\journal{Transportation Research Part C}

\begin{document}

\begin{frontmatter}

\title{Using Model Predictive Control To Reduce Traffic Emissions on Urban Freeways}

\author[DTU]{Alexander Hammerl}
\author[DTU]{Ravi Seshadri}
\author[DTU]{Thomas Kjær Rasmussen}
\author[DTU]{Otto Anker Nielsen}

\affiliation[DTU]{%
  organization={Department of Management, Engineering and Technology, Technical University of Denmark},
  addressline={Akademivej 358}, 
  city={Lyngby},
  postcode={2800}, 
  country={Denmark}
}

\begin{abstract}
Urban traffic congestion significantly impacts regional air quality and contributes substantially to pollutant emissions. Suburban freeway corridors are a major source of traffic-related emissions, particularly nitrogen oxides ($\text{NO}_x$) and carbon dioxide ($\text{CO}_2$). This paper proposes a Model Predictive Control (MPC) framework aimed at emission reduction on peripheral freeway corridors. Emission rates on freeways exhibit high sensitivity to speed fluctuations and congestion recovery processes. To address this relationship, we develop and analyze a bounded-acceleration continuum traffic flow model. By introducing an upper limit on vehicle acceleration capabilities, we enhance behavioral realism through the incorporation of driver responses to congestion, which is widely recognized as a main cause of the important capacity drop phenomenon. Our approach implements dynamically optimized variable speed limits (VSLs) at strategic corridor locations, balancing the dual objectives of minimizing both travel time and emissions as quantified by the COPERT V \cite{Emep2023} model. Numerical simulations demonstrate that this framework effectively manages congestion and reduces emissions across various traffic demand scenarios.
\end{abstract}


\begin{highlights}
\item We introduce a bounded-acceleration macroscopic traffic flow model that captures capacity drop phenomena.

\item We formulate an analytically tractable Model Predictive Control (MPC) framework to optimize Variable Speed Limits (VSL) on suburban freeways.

\item We integrate the COPERT V macroscopic emission model to simultaneously minimize travel time and pollutant emissions.

\item We demonstrate through simulations that our control strategy significantly reduces emissions with only modest travel time increases, even under demand uncertainty.
\end{highlights}


\end{frontmatter}



\section{Introduction}
Urban traffic congestion poses significant challenges to sustainable mobility, degrades air quality, and increases pollutant emissions. While traditional traffic management strategies primarily focus on enhancing mobility and reducing delays, emissions reduction is also important for sustainable urban traffic policies (cf. \cite{VanDenBerg2007}, \cite{Zegeye2011}, \cite{Zegeye2012}, \cite{Zegeye2013} for previous works). 
Variable Speed Limit (VSL) systems are dynamic traffic management tools that adjust posted speed limits in response to real-time traffic conditions through overhead or roadside display signs. The review paper by \cite{Khondaker2015} summarizes the operational benefits of traditional VSL applications as follows: improved safety,  prevention of traffic breakdown, and increase of throughput at bottlenecks. The latter two are possible due to the capacity drop phenomenon, where the maximum flow rate at an active bottleneck can drop below the theoretical capacity once congestion sets in (\cite{Banks1990}, \cite{Hall1991},\cite{Cassidy1999},\cite{Bertini2005}, \cite{Jin2017}) - a reduction that VSL systems can help prevent by strategically slowing down approaching traffic. Modern VSL control strategies (e.g. \cite{Hegyi2005}, \cite{Papageorgiou2008}, \cite{Carlson2010}, \cite{Chen2014}, \cite{Carlson2014},\cite{Jin2015}, \cite{Khondaker2015a}) are based on predictive modeling of traffic behavior. They employ a closed-loop feedback control mechanism to continuously update its predictions and control actions based on new traffic data. These predictions typically rely on either macroscopic or microscopic traffic flow models—such as derivatives of Payne’s model (\cite{Payne1971}, \cite{Whitham2011}), the LWR model, or car-following models ( e.g. \cite{Bando1995}, \cite{Treiber2000}, \cite{Helbing2001}, \cite{Kesting2010}).

Microscopic traffic models offer high resolution and detailed behavioral representations but require fine-grained input data and are computationally intensive, which can limit their utility for large-scale real-time control. In contrast, macroscopic models aggregate vehicle behavior into flow, density, and speed fields, offering superior scalability and robustness to input uncertainty. A similar distinction holds for emission modeling: while microscopic approaches (e.g. \cite{Ahn2008}) can capture engine dynamics in detail, macroscopic emission models - such as COPERT - offer tractable, speed-dependent aggregate emission rates and are widely adopted in policy contexts.

Higher-order macroscopic models attempt to incorporate dynamic driver behavior through reaction and anticipation terms. Earlier, yet still popular representatives of this model class have been the subject of critical debate due to problematic features such as faster-than-vehicle characteristic speeds, negative flows, and nonphysical shock structures (see \cite{Daganzo1995}, \cite{Helbing2009}). More recent formulations (see \cite{Aw2000}, \cite{Zhang2002} and extensions) resolve these deficiencies but their analytical complexity and purely phenomenological nature remains a limitation, especially for time-critical control applications.

In this work, we adopt a macroscopic modeling framework that extends the classical LWR theory by incorporating bounded acceleration constraints, extending the approach of \cite{Lebacque2002, Lebacque2003, Jin2018}. This Bounded Acceleration LWR (BA-LWR) model retains the computational efficiency and robustness of first-order theories while addressing their key deficiencies—most notably the unrealistic assumption of infinite acceleration. Unlike classical kinematic wave theories, the BA-LWR model provides an endogenous mechanism for the capacity drop phenomenon (see e.g. \cite{Jin2015a}, \cite{Jin2017b}), captures key second-order effects such as congestion recovery, and enables the derivation of analytical expressions for vehicle travel times evolution, which is an essential feature for predictive control design. Compared to classical higher-order models, our approach strikes a balance between behavioral realism, interpretability, and analytical tractability, making it well suited for implementation in a receding-horizon Model Predictive Control (MPC) framework. 

The rest of the paper is structured as follows: Section \ref{section:bamodel} presents the applied traffic flow model; Section \ref{section:infra_em} discusses modeling choices for road infrastructure and vehicle emissions; Section \ref{section:solution} formulates the MPC optimization problem and discusses results; Section \ref{section:numerical} presents numerical examples that validate the efficacy and robustness of the proposed control strategy; and Section \ref{section:conclusion} concludes with future research directions.

\section{A Bounded-Acceleration Traffic Flow Model}
\label{section:bamodel}
We employ a continuum approach to traffic flow modeling by conceptualizing traffic propagation as a continuous medium, rather than considering individual vehicle entities (see e.g. \cite{Jin2021}). The mathematical formulation of the model describes the relationship between the aggregate parameters traffic density k, flow rate q, and average velocity v, and proceeds as follows: The rate of change in the total number of vehicles contained in any road segment \( [x_1, x_2] \) where $x_2>x_1$ is equal to the net flow of vehicles out of the segment, i.e.
\begin{equation}
\frac{d}{dt} \int_{x_1}^{x_2} k(x,t) \, dx = - \left[ q(x,t) \right]_{x_1}^{x_2}.
\label{eq:integralconservation}
\end{equation}
If \( k \) and \( q \) are differentiable, \eqref{eq:integralconservation} can be transformed to obtain the partial differential equation
\begin{equation}
	\frac{\partial k}{\partial t} + \frac{\partial q}{\partial x} = 0.
	\label{eq:conservation}
\end{equation}
The seminal work in this field is the Lighthill-Whitham-Richards (LWR)(\cite{Lighthill1955},\cite{Richards1956}) theory, which assumes the existence of a functional relationship between q and k under differentiable conditions:

\begin{equation}
	q(x,t)=Q(x,k(x,t)),
	\label{eq:fundamental}
\end{equation}

where $Q$ is a concave, non-negative function that is equal to zero at $k=0$ and at the \textit{jam density} $k=k_j$. The third parameter, $v$, is given by the fundamental relationship $q = v \cdot k$. On substituting equation \ref{eq:conservation} into \ref{eq:fundamental}, we obtain 
\begin{equation}
	\frac{\partial k}{\partial t} + \frac{dQ}{dk} \cdot \frac{\partial k}{\partial x} = 0,
	\label{eq:differentiallwr}
\end{equation}
which defines a unique solution for $k(x,t)$ and $q(x,t)$ for given initial and boundary conditions if $q$ and $k$ are differentiable.
For solutions at discontinuities, additional entropy conditions are required to ensure uniqueness (see e.g. \cite{Lebacque1996}, \cite{Jin2009}), \cite{Holden2015}.  In this work, we employ the entropy solution 
\begin{equation}
Q(x,t) = \min \left[ D ( K(x-,t), x ), S ( K(x+,t), x ) \right],
\end{equation}
where $D(\kappa, x)$ and $S(\kappa,x)$ represent traffic demand and supply, respectively \cite{Lebacque1996}:
\begin{equation}
\label{eq:dem}
D(\kappa, x) =
\begin{cases} 
Q(\kappa, x-) & \text{if } \kappa \leq k_{\text{crit}}(x-) \\[10pt]
Q_{\max}(x-) & \text{if } \kappa \geq k_{\text{crit}}(x-),
\end{cases}
\end{equation}

\begin{equation}
\label{eq:sup}
S(\kappa, x) =
\begin{cases} 
Q_{\max}(x+) & \text{if } \kappa \leq k_{\text{crit}}(x+) \\[10pt]
Q(\kappa, x+) & \text{if } \kappa \geq k_{\text{crit}}(x+).
\end{cases}
\end{equation}
Here, $Q_{\text{max}}$ denotes the maximum value of $Q(k)$ and $k_{\text{crit}}(x)$ denotes the critical density such that $Q(k_{\text{crit}})=Q_{\text{max}}$. This condition can be formulated equivalently: A discontinuity in the initial value conditions leads to the formation of a shock wave when velocity decreases, while at increasing velocity results in a rarefaction fan \cite{Lebacque2003}. Lebacque (\cite{Lebacque2002}, \cite{Lebacque2003}) overcomes the deficiencies of potentially infinite accelerations in the original LWR theory by introducing a constant upper limit for the acceleration of a vehicle. This work extends Lebacque's theory by considering a vehicle's maximum acceleration not as a fixed value, but as a function of the current velocity, $A(v)$. The constraint $A'(v)<0$ is common in traffic flow modeling; however, it is not necessary for the validity of our derivations. For the convective acceleration of a moving vehicle, \( \frac{dv}{dt} \), the relationship \( \frac{dv}{dt} = \frac{\partial v}{\partial t} + \frac{dx}{dt} \cdot \frac{\partial v}{\partial x} = \frac{\partial v}{\partial t} + v \cdot \frac{\partial v}{\partial x} \) applies. The resulting two-phase model ("equilibrium" and "bounded acceleration") can therefore be characterized by the following equation:

\begin{equation}
\label{equation:ba}
\left\{
\begin{array}{l}
v(x,t) = V(k(x,t)) \\[10pt]
v(x,t) \leq V(k(x,t))
\end{array}
\right.
\quad \text{if} \quad
\begin{array}{l}
\frac{\partial v}{\partial t} + v \frac{\partial v}{\partial x} \leq A(v) \quad \text{EQ phase} \\[10pt]
\frac{\partial v}{\partial x} + v \frac{\partial v}{\partial x} = A(v) \quad \text{BA phase}
\end{array}
\end{equation}

The BA phase of the system can be formulated as a system of hyperbolic conservation laws with relaxation (see e.g. \cite{Liu1987}, \cite{Jin2001}) as follows:
\begin{equation}
\label{eq:conservative}
\partial_t U + \partial_x \Phi(U) = S(U), \quad
S(U) = \begin{pmatrix} 0 \\ a(V) \end{pmatrix},
\end{equation}
where the conserved variables and flux function are given by
\[
U = \begin{pmatrix} K \\ V \end{pmatrix}, 
\Phi(U) = \begin{pmatrix} K V \\ \frac{V^2}{2} \end{pmatrix}.
\]

The Jacobian of the flux $\Phi$ is
\[
D\Phi(U) = 
\begin{pmatrix}
\frac{\partial (KV)}{\partial K} & \frac{\partial (KV)}{\partial V} \\
0 & \frac{\partial (V^2/2)}{\partial V}
\end{pmatrix}
=
\begin{pmatrix}
V & K \\
0 & V
\end{pmatrix}.
\]

Since this matrix is upper triangular, its eigenvalues are given by its diagonal entries:
\[
\lambda_1(U) = \lambda_2(U) = V.
\]

In line with the standard interpretation in traffic flow theory (see, e.g., \cite{Daganzo1995}, \cite{Helbing2009}), these characteristic speeds describe the propagation velocity of small perturbations. In this case, perturbations travel at the same speed as the vehicles themselves.

Let \( x_0(t) \) be a boundary between an EQ and a BA phase according to \eqref{equation:ba}, such that homogeneous congestion conditions exist to the left of the boundary:
\[
k(x,t) = k^+, \quad q(x,t) = Q(k^+) \quad \text{for all } (x,t) \text{ with } x \leq x_0(t),
\]
where \( k^+ > k^{\text{crit}} \).  Let $r(\cdot):=x_0'(t)$ denote the propagation speed of the phase boundary. The property presented in Theorem \ref{theorem:qklinearity} is helpful for the construction of analytical solutions and was demonstrated by \cite{Lebacque2003} for constant acceleration curves $A(v)=A_0$. In the following, we demonstrate its validity for arbitrary acceleration functions $A(v) \in C^1$.

\begin{theorem}
\label{theorem:qklinearity}
The parametric curve along a vehicle trajectory starting at \( x_0(t) \) in the \( q \)-\( k \) phase plane always forms a straight line with slope \( -x_0'(t) \).
\end{theorem}

Consider two virtual vehicles A (leader) and B (follower): vehicle A occupies the boundary point $(x,t)=(x_0(t),t)$ and travels with speed $v(t)$. Vehicle B starts one (infinitesimal) time–headway $\tau>0$ later at $(x_0(t-\tau),t-\tau)$,  with the \emph{same} initial speed $v(t-\tau)$. For this leader-follower pair define
\begin{equation}\label{eq:P1}
  \begin{aligned}
    \Delta x_\tau(t)  &:= c\,\tau + \int_{t-\tau}^{t} v(s)\,ds,
         &\quad&\text{(spacing)}\\[2pt]
    \Delta v_{H,\tau}(t) &:= v(t)-v(t-\tau),
         &\quad&\text{(headway speed-difference)}\\[2pt]
    k_\tau(t)        &:= \dfrac{1}{\Delta x_\tau(t)},
         &\quad&\text{(density)}\\[6pt]
    q_\tau(t)        &:= k_\tau(t)\,v(t),
         &\quad&\text{(flow)}
  \end{aligned}
\end{equation}

where $c:=\frac{x_0(t)-x_0(t-\tau)}{\tau}$ denotes the average slope of the release boundary between $t - \tau$ and $t$. Derivatives with respect to $t$ are written $\dot v(t)$ and $\ddot v(t)$.

\begin{lemma}
For every fixed $t$,
\begin{align}
  \Delta v_{H,\tau}(t) &= \dot v(t)\,\tau - \tfrac12\ddot v(t)\,\tau^{2}+O(\tau^{3}),
  \label{eq:L1.1}\\
  \Delta x_\tau(t)     &= [c+v(t)]\,\tau
                          -\tfrac12\dot v(t)\,\tau^{2}
                          +O(\tau^{3}),
  \label{eq:L1.2}\\
  k_\tau(t)            &= \dfrac{1}{\tau\,[c+v(t)]}+O(1).          \label{eq:L1.3}
\end{align}
\end{lemma}
\begin{proof}
We approximate $v(t-\tau)$ through a Taylor series expansion around $t$:
\[
  v(t-\tau)=v(t)-\tau\dot v(t)+\tfrac{\tau^{2}}{2}\ddot v(t)
            -\tfrac{\tau^{3}}{6}v^{(3)}(\xi_{1}),
            \qquad \xi_{1}\in(t-\tau,t).
\]
Hence
\[
  \Delta v_{H,\tau}(t)=v(t)-v(t-\tau)
            =\tau\dot v(t)-\tfrac{\tau^{2}}{2}\ddot v(t)+O(\tau^{3}),
\]
which matches~\eqref{eq:L1.1}.

To obtain the expression for the spacing $\Delta x_\tau(t)$, we apply Taylor’s theorem under the integral:
\begin{align*}
\int_{t-\tau}^{t}\!v(s)\,ds
  &= \int_{0}^{\tau}\!v(t-\sigma)\,d\sigma \\
  &= \tau v(t) - \tfrac{\tau^{2}}{2}\dot v(t)
     + \tfrac{\tau^{3}}{6}\ddot v(t)
     - \tfrac{\tau^{4}}{24}v^{(3)}(\xi_{2}),
     \quad \xi_{2} \in (t-\tau, t).
\end{align*}
By adding the spatial headway at the start time of B, $c\tau$, we obtain~\eqref{eq:L1.2}.

Write the spacing in the separated form
\begin{align*}
\Delta x_\tau(t) &= (c + v(t))\,\tau \left[ 1 - \frac{\dot{v}(t)}{2(c + v(t))} \tau + \frac{\ddot{v}(t)}{6(c + v(t))} \tau^2 + O(\tau^3) \right] \\
                 &= (c + v(t))\,\tau \left[ 1 - \varepsilon_\tau(t) \right],
\end{align*}
where $\displaystyle
      \varepsilon_{\tau}=\tfrac{\tau}{2}\frac{\dot v(t)}{c+v(t)}+O(\tau^{2})$.
Using the geometric-series identity
$(1-\varepsilon)^{-1}=1+\varepsilon+\varepsilon^{2}+O(\varepsilon^{3})$
gives
\[
  \frac{1}{\Delta x_\tau(t)}
    =\frac{1}{\tau\,[c+v(t)]}\,
      \bigl(1+\varepsilon_{\tau}+O(\tau^{2})\bigr)
    =\frac{1}{\tau\,[c+v(t)]}
      +\frac{\dot v(t)}{2\,[c+v(t)]^{2}}
      +O(\tau).
\]
This establishes~\eqref{eq:L1.3} and completes the lemma.
\end{proof}

\begin{lemma}
For every $\tau>0$,
\begin{equation}\label{eq:L2}
  \frac{d q_\tau}{d k_\tau}(t)
  \;=\;
  v(t) \;-\;
  \frac{\dot v(t)}{k_\tau(t)\,\Delta v_{H,\tau}(t)}.
\end{equation}
\end{lemma}

\begin{proof}
By definition, \( k_\tau(t) = 1 / \Delta x_\tau(t) \) with \( \Delta x_\tau(t) \) given in~\eqref{eq:P1}. Because \( \dot{\Delta x}_\tau(t) = v(t) - v(t - \tau) =: \Delta v_\tau(t) \),
    \begin{equation}
      \dot{k}_\tau(t)
        = -\frac{\dot{\Delta x}_\tau(t)}{(\Delta x_\tau(t))^{2}}
        = -k_\tau(t)^{2} \cdot \Delta v_\tau(t).
      \label{eq:kdot}
    \end{equation}
Differentiating \( q_\tau = k_\tau \, v \) with respect to \( t \) yields

    \begin{equation}
      \dot{q}_\tau(t)
        = \dot{k}_\tau(t)\,v(t)\;+\;k_\tau(t)\,\dot{v}(t).
      \label{eq:qdot}
    \end{equation}
The trajectory \( t \mapsto (k_\tau(t), q_\tau(t)) \) defines a parametrized curve in the \( k \)–\( q \) plane. By the chain rule, its derivative with respect to \( k \) is given by
    \begin{equation}
      \frac{dq_\tau}{dk_\tau}
        = \frac{\dot{q}_\tau(t)}{\dot{k}_\tau(t)}
        \quad \text{provided } \dot{k}_\tau(t) \neq 0.
      \label{eq:chain}
    \end{equation}
Substituting \eqref{eq:kdot} and \eqref{eq:qdot} into \eqref{eq:chain}:
    \[
  \frac{dq_\tau}{dk_\tau}
    = \frac{[-k_\tau(t)^2 \cdot \Delta v_\tau(t)] \cdot v(t)
             + k_\tau(t) \cdot \dot{v}(t)}
           { -k_\tau(t)^2 \cdot \Delta v_\tau(t)}
    = v(t) - \frac{\dot{v}(t)}{k_\tau(t) \cdot \Delta v_\tau(t)},
\]
    which is exactly formula~\eqref{eq:L2}.
\end{proof}

\begin{proof}[Proof of theorem \ref{theorem:qklinearity}]
from equations \ref{eq:L1.1}–\ref{eq:L1.3} we establish that \[k_\tau\,\Delta v_{H,\tau} =
    k_\tau\,\Delta v_{H,\tau}
        = \frac{\dot v(t)}{c+v(t)} + O(\tau)
        \quad\text{as }\tau\to0^{+}.
  \]
  Substituting this into~\eqref{eq:L2} gives
  \[
    \frac{d q_\tau}{d k_\tau}(t)
      = v(t)\;-\;\bigl[c+v(t)\bigr] + O(\tau)
      = -c + O(\tau).
  \]
  Taking $\tau\to0^{+}$ proves the claimed limit. Integrating with respect to $k$ yields the straight line announced in the statement, with an intercept determined by the $t$-dependent limit of $q_\tau + c\,k_\tau$.
\end{proof}

For transitions from the equilibrium phase to the bounded-acceleration phase as desribed in theorem \ref{theorem:qklinearity}, that is, when vehicles accelerate out of a congested region, \cite{Lebacque2003} defines the following 
additional boundary condition: The acceleration wave propagates with a reaction delay \( \theta \) in the opposite direction of travel from one vehicle to the preceding one. Since the distance between two vehicles is proportional to the inverse 
traffic density \( \frac{1}{k} \), the formula 
\[
\rho(k) = V(k) - \frac{q_0}{k}
\]
results for the propagation speed of the acceleration wave. Following the recommendation of \cite{Lebacque2002}, we set 
\( q_0 = Q_{\max} \), so that \( \rho(k) \leq 0 \) and consequently \( d\xi/dt \leq 0 \) always holds.

To analyze the effects of changes in the variable speed limit on traffic flow behavior, it is also helpful to characterize the behavior of $q(t)$ at the VSL location $x_0$, starting from the release boundary. The following theorem derives a formula for $q(t)$ under these conditions:

\begin{theorem}
\label{theroem:statobs}
Consider a phase transition from LWR equilibrium to BA phase with boundary trajectory \( x_b(t) \). Fix a point \( (x_0, t_0) \) on this boundary. Let \( q_1 \) and \( v_1 \) denote the upstream equilibrium flow and speed, respectively,\( V_e \) the final speed in the BA phase, and \( D_e \) the distance required for a vehicle to accelerate from \( v_1 \) to \( V_e \).
 Define
\[
t_{+} := t_0 + \frac{D_e}{r}, \qquad
q_0 := \frac{q_1}{1 + r / V_e}.
\]
Then the flow observed by a stationary detector at \( x_0 \) is given by
\[
q(x_0, t) =
\begin{cases}
\displaystyle \frac{q_1}{1 + r / v(t)}, & t_0 \le t < t_{+}, \\[10pt]
q_0, & t \ge t_{+}.
\end{cases}
\]
\end{theorem}

\begin{proof}
Consider two consecutive vehicles departing from the boundary at times \( t_d \geq t_0 \) and \( t_d + \Delta t_0 \), where \( \Delta t_0 \) represents the release headway. Let \( r \) denote the average speed of the boundary trajectory during this interval.

The road distances these vehicles must traverse to reach \( x_0 \) are
\[
D(t_d) = r(t_d - t_0), \quad D(t_d + \Delta t_0) = D(t_d) + r\Delta t_0.
\]

Since vehicle trajectories departing from the boundary are translation-invariant, the travel time to reach distance \( D \) depends only on \( D \) itself. Denote this travel time function by \( T(D) \), where \( T'(D) = \frac{1}{v(D)} \) by definition of speed.

The arrival times at \( x_0 \) are
\begin{align}
t_{\text{arr},1} &= t_d + T(D(t_d)), \\
t_{\text{arr},2} &= t_d + \Delta t_0 + T(D(t_d) + r\Delta t_0).
\end{align}

Applying Taylor expansion to \( T \) around \( D(t_d) \):
\[
T(D(t_d) + r\Delta t_0) = T(D(t_d)) + T'(D(t_d)) \cdot r\Delta t_0 + O(\Delta t_0^2).
\]

The arrival headway is therefore
\[
\Delta T(t_d) = t_{\text{arr},2} - t_{\text{arr},1} = \Delta t_0(1 + rT'(D(t_d))) + O(\Delta t_0^2).
\]

Since the boundary releases vehicles at rate \( q_1 = 1/\Delta t_0 \), the instantaneous flow at the detector is
\begin{align}
q(t_d) &= \frac{1}{\Delta T(t_d)} = \frac{1}{\Delta t_0(1 + rT'(D(t_d)) + O(\Delta t_0))} \\
&= \frac{q_1}{1 + rT'(D(t_d)) + O(\Delta t_0)} = \frac{q_1}{1 + r/v(D(t_d))},
\end{align}
where the final equality uses \( T'(D) = 1/v(D) \) and neglects higher-order terms as \( \Delta t_0 \to 0 \).
\end{proof}

The trajectory in the q-k phase plane observed from the stationary position $x_0$ follows from the following corollary:

\begin{corollary}
Let an acceleration wave pass over a fixed detector located at position \(x_0\). Then, the trajectory traced in the \((k, q)\)-plane is a straight line with slope \( -r \).
\end{corollary}

\begin{proof}
We begin with the relation
\[
q = \frac{q_1}{1 + r/v},
\]
and solve for the velocity \(v\):
\[
v = \frac{r\,q}{q_1 - q}.
\]
Substituting this expression for \(v\) into the identity \(k = q/v\), we obtain
\[
k = \frac{q_1 - q}{r}.
\]
Solving for \(q\) yields
\[
q = q_1 - r\,k,
\]
which holds for \(q \in [q_0, q_1]\). This completes the proof.
\end{proof}

\section{Infrastructure, Demand, and Emissions}
\label{section:infra_em}
The geometric properties of the analyzed corridor are presented schematically in Figure \ref{figure:schematic}.

\begin{figure}[H]
    \centering
    \includegraphics[width=\linewidth]{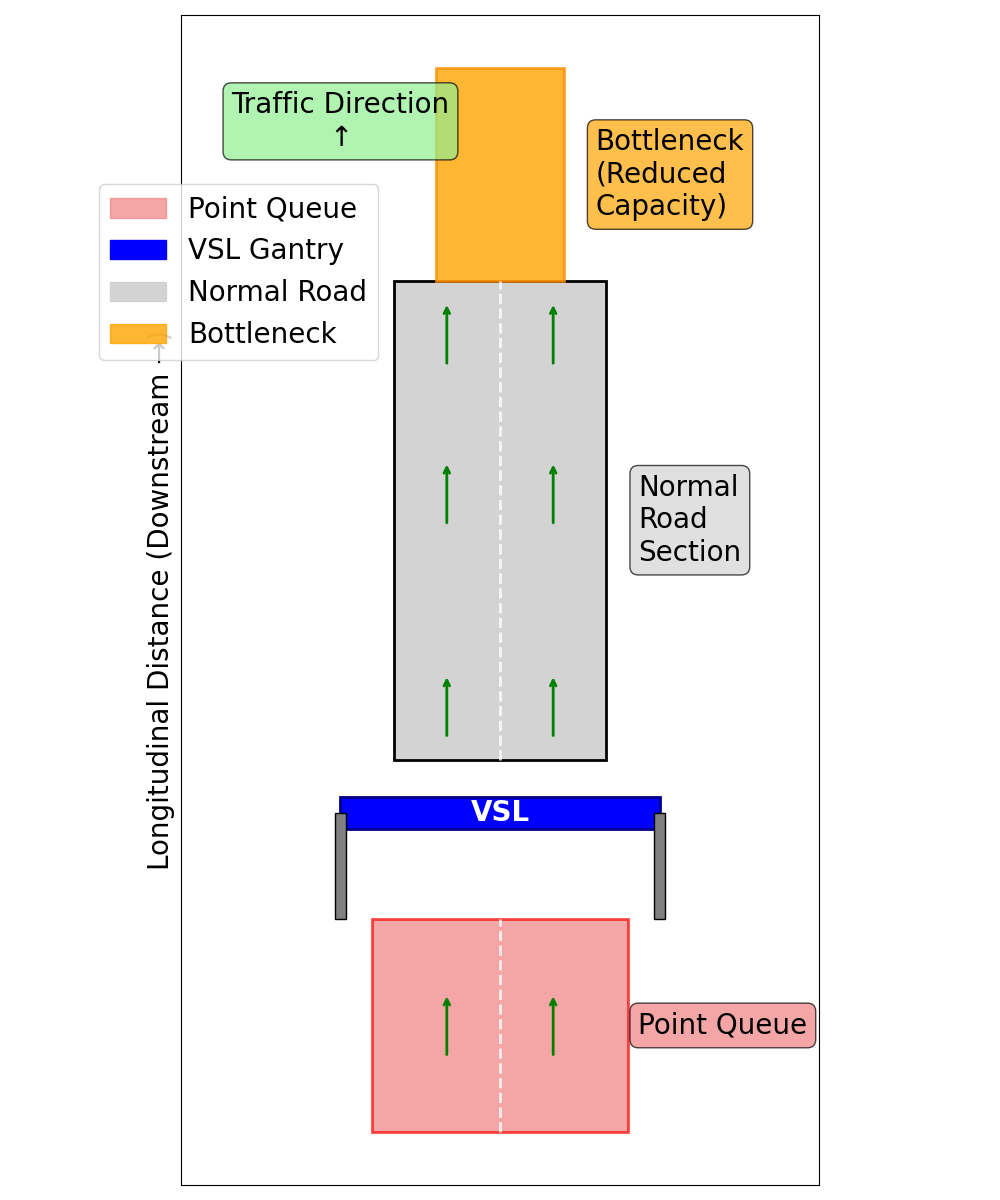}
    \caption{Schematic illustration of the corridor and point queue.}
    \label{figure:schematic}
\end{figure}

We consider the VSL to be located at the upstream boundary of the corridor, i.e., at position \( x = 0 \). The inflow at this boundary follows a point queue mechanism (\cite{Jin2015c}): if the available supply is insufficient for immediate entry, vehicles are held in an abstract upstream queue. Let \( \ell(t) \) denote the length of this point queue at time \( t \). Its dynamics are described by the conservation equation
\begin{equation}
\label{eq:pq_conservation}
\frac{d\ell(t)}{dt} = D(t) - q(0,t),
\end{equation}
where the actual inflow \( q(0,t) \) is determined by the state-dependent rule
\begin{equation}
\label{eq:pq_demsup}
q(0,t) =
\begin{cases}
s(k(0,t)), & \text{if } \ell(t) > 0, \\
\min\left\{ s(k(0,t)), D(t) \right\}, & \text{if } \ell(t) = 0.
\end{cases}
\end{equation}

For passenger cars with petrol engines compliant with the Euro 5 emission standard, the COPERT V Tier 3 model expresses hot exhaust emissions of CO, NO$_x$, and HC (VOC) as speed-dependent functions of the form:
\[
E(v) = \frac{\alpha v^2 + \beta v + \gamma + \delta / v}{\varepsilon v^2 + \zeta v + \eta},
\]
where \(v\) is the average vehicle speed in km/h, and the parameters \(\alpha, \beta, \gamma, \delta, \varepsilon, \zeta, \eta\) are pollutant-specific coefficients obtained from the EMEP/EEA Guidebook Appendix 4.

The fitted coefficients for Euro 5 petrol passenger cars are:

\begin{align}
E_{\text{CO}}(v) &= \frac{0.000445 \, v^2 - 0.102076 \, v + 6.876928 + \frac{10.383849}{v}}{0.001621 \, v^2 - 0.437563 \, v + 30.337333}, \\
E_{\text{NO}_x}(v) &= \frac{-0.000315 \, v^2 + 0.103057 \, v + 0.239057 - \frac{0.339279}{v}}{0.034536 \, v^2 + 1.986013 \, v + 1.263763}, \\
E_{\text{HC}}(v) &= \frac{0.000004 \, v^2 - 0.000707 \, v + 0.045249 + \frac{0.173074}{v}}{0.000070 \, v^2 - 0.047538 \, v + 6.212053}.
\end{align}

All emission factors \(E(v)\) are in units of grams per kilometer [g/km]. These expressions are valid across a broad range of speeds typically encountered in real-world traffic (\(v \in [10, 130]\,\text{km/h}\)), including urban streets and peripheral freeways. 

\section{Solution Of The Control Problem}
\label{section:solution}
In this section, we solve the control problem resulting from the system behavior derived in section \ref{section:bamodel}. We assume a triangular shape of the fundamental diagram \( q(k) \):
\[
q(k) = \min\{v_f \cdot k,\, w \cdot (k_j - k)\},
\]
where \( v_f \) denotes the free-flow speed and \( w \) the wave speed under congested conditions. Furthermore, a constant maximum acceleration \( A_0 \) during the BA phase is assumed.
In the optimization of the VSL trajectory, the running cost per vehicle, i.e., the emissions as a function of travel time~\(\tau\), is given by
\begin{equation}
\label{eq:runningcost}
C(\tau) = \mu_1 \tau + \mu_2 \left[ \lambda_1 E_{\mathrm{CO}}\left(\frac{l}{\tau}\right) + \lambda_2 E_{\mathrm{NO}_x}\left(\frac{l}{\tau}\right) + \lambda_3 E_{\mathrm{HC}}\left(\frac{l}{\tau}\right) \right],
\end{equation}
where \(\mu_1, \mu_2\) and \(\lambda_1, \lambda_2, \lambda_3\) are weighting parameters. The normalization conditions \(\mu_1 + \mu_2 = 1\) and \(\lambda_1 + \lambda_2 + \lambda_3 = 1\) apply. 

Steep VSL increases can generate shock waves when vehicles accelerating toward the current VSL overtake slower preceding vehicles.  To prevent such collisions, we introduce a condition whereby a vehicle starting at time $t+dt$ always arrives later than one starting at time $t$. This yields a Bardos-Leroux-Nédélec (BLN) boundary condition:
\[
    \left[t + dt + \tau(t + dt)\right] - \left[t + \tau(t)\right] > 0 \quad \Longrightarrow \quad 1 + \frac{d\tau}{dt} \geq 0.
\]
Equivalently:
\begin{equation}
\label{eq:BLNraw}
v'_{\text{VSL}}(t) \leq \frac{\left( 1 - \frac{v_{\text{VSL}}(t) - v(t)}{A_0 \, v_{\text{VSL}}(t)} v'(t) \right)}{\frac{v_{\text{VSL}}^2(t) - v^2(t) - 2A_0 l}{2A_0 \, v_{\text{VSL}}^2(t)}}.
\end{equation}
For analytical convenience, we employ the upper bound:
\begin{equation}
\label{eq:BLNapprox}
v'_{\text{VSL}}(t) \leq \frac{2A_0 \, v_{\text{VSL}}(t)}{2A_0 l - \left( v_{\text{VSL}}^2(t) - v^2(t) \right)}.
\end{equation}

For an increase in the speed limit from \( v_c \) to \( v_c^+ \), the dynamic equations of vehicle kinematics in the BA phase are:
\[
\frac{dv}{dt} = A_0, \qquad v(0) = v_c.
\]
The backward speed of acceleration wave is \( r \). Using the chain rule, the transformation from a time-dependent to a space-dependent ordinary differential equation follows:
\[
\frac{dv}{dx} = \frac{dv/dt}{dx/dt} = \frac{A_0}{v}.
\]
Separation of variables leads to:
\[
v \, dv = A_0 \, dx.
\]
Integration of both sides yields:
\[
\int_{v_c}^{v} v' \, dv' = \int_0^{x} A_0 \, dx' \quad \Rightarrow \quad \frac{v^2 - v_c^2}{2} = A_0 x.
\]

A vehicle leaving the release boundary at time \( t \) covers the distance \( D = r(t - t_0) \) to the VSL position. Thus:
\[
v(D) = \sqrt{v_c^2 + 2A_0 D} = \sqrt{v_c^2 + 2A_0 r(t - t_0)}.
\]

Substituting into the formula derived from theorem \ref{theroem:statobs} yields:
\[
q(t) =
\begin{cases}
\displaystyle \frac{q_0}{1 + \dfrac{r}{\sqrt{v_c^2 + 2A_0 r(t - t_0)}}}, & \text{for } t_0 \le t < t_{+}, \\[10pt]
q_1, & \text{for } t \ge t_{+},
\end{cases}
\]
where $q_0=q(x_0,t_0)$ and $q_1$ denotes the downstream equilibrium flow value derived from theorem \ref{theorem:qklinearity}. Figure \ref{figure:transitions} demonstrates the change in equilibrium flow at the VSL position during a stepwise increase in the speed limit from 50 km / h to 85 km / h to the free flow speed of 120 km/h. Although the same boundary conditions as in \cite{Lebacque2003} are used for the propagation of acceleration waves, the present model can also exhibit such waves between two uncongested phases, as the diagram shows.

\begin{figure}[H]
  \centering
  \includegraphics[width=0.6\textwidth]{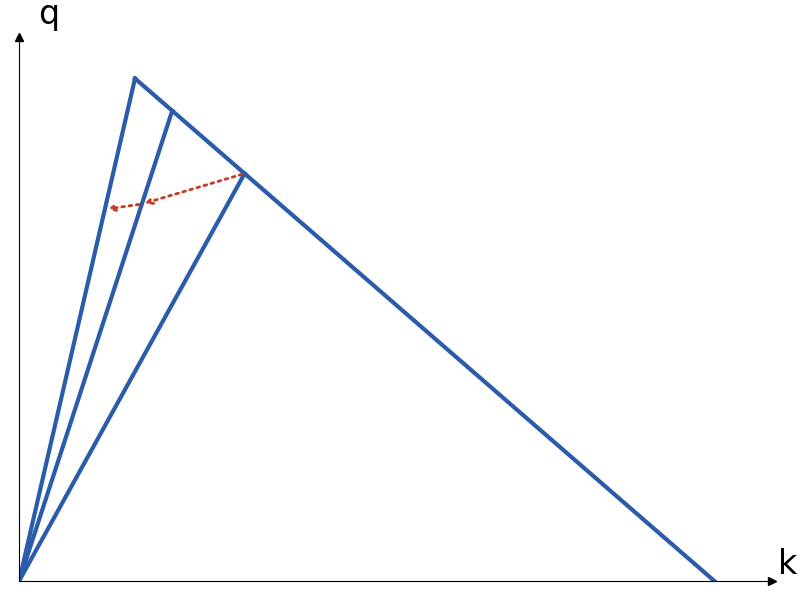}
  \caption{Transition of equilibrium states with twofold speed increase. Fundamental diagram (blue) and acceleration-induced phase transitions (red).}
  \label{figure:transitions}
\end{figure}

We study a single freeway segment consisting, in the direction of traffic flow, of (i) a point queue of length $\,l_q(t)$, followed by (ii) a VSL sign located at $x=0$, and (iii) a free-flow section of fixed length $L$ that starts at the sign and ends at the downstream exit of the segment. Vehicles join the queue with exogenous demand \( r(t) \) and leave the segment at flow rate \( f_\mathrm{exit}(t) \). The independent variable \( t \) denotes the departure time; the running cost is accumulated when a vehicle exits the segment. The posted limit is the state
\[
v(t) = v_\mathrm{VSL}(t) > 0, \qquad \dot v(t) = a(t) \in \mathcal{U}(t) \subset \mathbb{R} \quad \text{(control)},
\]
where \( a(t) \) is the rate of change. Upward changes are restricted by the path constraint \ref{eq:BLNapprox}, while deceleration is assumed
instantaneous, so \( a(t) \) is unbounded below. The traffic state is characterized by four variables:
\begin{itemize}
\item \( l_q(t) \) – the queue length upstream of the VSL sign in vehicles.
\item \( x_r(t) \le 0 \) – the position of the \textit{release boundary}: when
\( v(t) \) increases, drivers start to accelerate upstream of the
sign; \( x_r(t) \) moves backward with speed determined by the
\( q \)--\( k \) slope below.
\item \( (k_0(t), q_0(t)) \) – the density and flow precisely at the VSL sign. Between two limit changes they travel along the affine line
\( q_0 = q^\mathrm{old} - r k_0 \) until they re‑intersect the
triangular fundamental diagram (FD).
\item \( w_q(t) \) – \textit{virtual waiting time}: the time that the
vehicle currently at the head of queue (HoQ) has spent in the
system up to \( t \).
\end{itemize}

Define the FD branch governed by the limit as
\[
q_\mathrm{cap}(v) = k_j\,\frac{v w}{v + w}, \qquad
\phi(v) = q_\mathrm{cap}(v),
\]
where \( w \) and \( k_j \) are standard triangular-diagram parameters.


\begin{equation}\label{eq:OC-full}
\begin{aligned}
\min_{\,v(\cdot),\,a(\cdot)}\quad
    & J = \int_{0}^{T_f} q_0(t)\,
           C\!\bigl(w_q(t)+\tau_f(t)\bigr)\,dt \\[1em]
\text{s.t.}\quad
& \dot l_q(t) = r(t) - q_0(t), \\ 
& l_q(0) = 0, \\[0.5em]
& \dot w_q(t) = 1 - \frac{q_0(t)}{r(t - w_q(t))}, \\
& w_q(0) = 0, \\[1em]
& \text{\bf Acceleration phase} \quad (t_0 \le t < t_{+}): \\[0.3em]
& \quad \dot q_0(t) =
    \frac{A_0\,q_s\,r^{2}}
         {\left[r + \sqrt{v_c^{2} + 2A_0 r(t - t_0)}\right]^{2}
          \sqrt{v_c^{2} + 2A_0 r(t - t_0)}}, \\[0.5em]
& \quad \dot k_0(t) = -\frac{1}{r}\,\dot q_0(t), \\[1em]
& \text{\bf Equilibrium phase} \quad (t \ge t_{+}): \\[0.3em]
& \quad \dot q_0(t) = 0, 
   \quad \dot k_0(t) = 0, \\[1em]
& a(t) = \dot v(t), 
   \quad a(t) \le 
   h\bigl(v(t)\bigr) := 
   \frac{2A_0\,v(t)^{2}}{2A_0 L - [v(t)^{2} - u_0(t)^{2}]}, \\[0.5em]
& v_{\min} \le v(t) \le v_{\max}, \\[1em]
& u_0(t) = \frac{q_0(t)}{k_0(t)}, \\[0.5em]
& \tau_f(t) =
  \begin{cases}
    \displaystyle
    \frac{v(t) - u_0(t)}{A_0}
    + \frac{L - \dfrac{v(t)^2 - u_0(t)^2}{2A_0}}{v(t)},
    & \text{if } v(t) > u_0(t), \\[1em]
    \displaystyle
    \frac{L}{v(t)}, & \text{if } v(t) \le u_0(t).
  \end{cases}
\end{aligned}
\end{equation}

Here,  \(A_0,L,v_c,q_s,r,t_0,t_{+}\) are additional constants from the triangular fundamental diagram and the most recent speed-increase episode.

Equation \eqref{eq:OC-full} contains \emph{all} dynamic relations needed for optimisation; the cost integrand uses the composite travel time \(\tau_f+w_q\), where \(\tau_f\) is computed with finite acceleration whenever \(v(t)\) exceeds the current traffic speed \(u_0(t)\).

Initial data \( l_q(0) = x_r(0) = 0 \) and \( (k_0(0), q_0(0)) \) on the FD close the forward problem. The penalty per vehicle is \( C(\tau) \), Hence
\begin{equation}\label{eq:J}
J[v(\cdot)] = \int_{0}^{T_f} F(t)\, C\!\left(w_q(t) + \tfrac{L}{v(t)}\right)\, \mathrm{d}t.
\end{equation}
\( T_f \) is chosen large enough that \( l_q(T_f) = 0 \).

The control program is solved in a Pontryagin framework with mixed state-control consstraints: we introduce costates \( \lambda = (\lambda_{l}, \lambda_{x_r}, \lambda_{k},
\lambda_{q}, \lambda_{w})^{\mathsf{T}} \) for the traffic states and
\( \lambda_v \) for \( v \). A non‑negative multiplier \( \mu(t) \) enforces the inequality \eqref{eq:BLNapprox}. The augmented Hamiltonian is
\begin{align*}
\widetilde H &= F\, C\!\left(w_q + \tfrac{L}{v}\right)
+ \lambda^{\mathsf{T}} F_x
+ \lambda_v a
+ \mu \bigl(a - h(x,v)\bigr),
\end{align*}
where \( F_x \) stacks the right‑hand sides of \eqref{eq:OC-full}.
By Gamkrelidze's extension of Pontryagin's Maximum Principle, the first‑order necessary conditions are as follows:
\begin{enumerate}[label=\textit{(\roman*)}]
\item State ODEs \eqref{eq:OC-full} with \( a = \dot v \).
\item Costate equations
\( \dot\lambda = -\partial \widetilde H / \partial x \),
\( \dot\lambda_v = -\partial \widetilde H / \partial v \), with terminal
conditions \( \lambda(T_f) = 0 \), \( \lambda_v(T_f) = 0 \).
\item \emph{Variational inequality (Euler condition):} 
      The optimal speed-limit profile \(v^{*}(t)\), with
      derivative \(a^{*}(t) = \dot{v}^{*}(t)\), satisfies
      \[
        0 \;\in\;
        \left.\frac{\partial \widetilde{H}}{\partial a}
        (t, x, v, a)\right|_{(x, v^{*}, a^{*})}
        \;+\;
        N_{\,(-\infty,\,h(x(t), v^{*}(t))]}
          \bigl(\dot{v}^{*}(t)\bigr),
      \]
      where \(N_S(\bar{a})\) denotes the normal-cone to the set \(S\) at the point \(\bar{a}\). Equivalently,
      \[
        \left.\frac{\partial \widetilde{H}}{\partial a}
        \right|_{a = \dot{v}^{*}(t)}
        \bigl(w - \dot{v}^{*}(t)\bigr) \;\ge\; 0,
        \qquad
        \forall\, w \le h\bigl(x(t), v^{*}(t)\bigr).
      \]
      Because \(\widetilde{H}\) is affine in \(a\), this reduces to:  
      \begin{itemize}
        \item a \textit{boundary arc}, where 
              \(\dot{v}^{*}(t) = h\bigl(x(t), v^{*}(t)\bigr)\)
              whenever \(\partial_a \widetilde{H} > 0\);
        \item an \textit{interior (holding) arc}, where 
              \(\dot{v}^{*}(t) = 0\) when \(\partial_a \widetilde{H} = 0\).
      \end{itemize}

\item \emph{Complementarity and sign conditions:} 
      \begin{equation}
\begin{aligned}
  \mu(t) &\ge 0, \\
  g\bigl(x(t), v^{*}(t), \dot{v}^{*}(t)\bigr) &\le 0, \\
  \mu(t) \cdot g\bigl(x(t), v^{*}(t), \dot{v}^{*}(t)\bigr) &= 0,
  \quad \text{for all } t \in [0, T_f].
\end{aligned}
\end{equation}
\end{enumerate}

\subsection{Model–Predictive Control (MPC) realisation}
\label{sec:mpc}

The finite–horizon optimal–control problem \eqref{eq:OC-full},
together with the Pontryagin–type first-order conditions, can be embedded in a
receding-horizon loop.  At each sampling instant
$t_k=k\,\Delta T$ the current traffic state
\(
  x_k=(l_q, w_q, q_0, k_0, v)^{\!\top}(t_k)
\)
is measured or estimated; a short-horizon instance of
\eqref{eq:OC-full} (prediction window \(T_p\)) is solved; and the
resulting optimal control signal \(v^{*}\) is applied only for the next
interval \([t_k,t_{k+1})\).
The procedure is summarised in Algorithm~\ref{alg:mpc}.

\begin{algorithm}[H]
\caption{MPC for bounded-gradient VSL}
\label{alg:mpc}
\begin{algorithmic}[1]
\Require
  sampling time $\Delta T>0$, prediction horizon $T_p\gg\Delta T$,
  model parameters $(A_0,w,k_j,L)$, \\
  arrival-rate forecast $r(t)$ on $[t_k,t_k+T_p]$.
\State initialise $k\gets0$, measure initial state $x_0$,
       set $t_0\gets0$
\While{$t_k<T_{\mathrm{f}}$ \textbf{and} $l_q(t_k)>0$}
    \Statex\hspace{1em}\rule{\linewidth-3em}{.4pt}
    \State \textbf{Solve short-horizon OCP}
           \[
             \min_{v(\cdot),\,a(\cdot)}\;
               \int_{t_k}^{t_k+T_p} 
                 q_0(t)\,C^{\!*}\!\bigl(w_q(t)+\tau_f(t)\bigr)\,dt
           \]
           subject to the state equations \eqref{eq:OC-full},
           BLN gradient cap $a\le h(x,v)$,
           and initial condition $x(t_k)=x_k$ \label{line:solve}.
    \State \textbf{Apply} the first portion of the optimal profile:
           $v(t)\gets v^{*}(t)$ for $t\in[t_k,t_{k+1})$,
           where $t_{k+1}=t_k+\Delta T$.
    \State \textbf{Update} the state by simulating
           \eqref{eq:OC-full} over $[t_k,t_{k+1})$
           with the applied $v(\cdot)$,
           obtaining $x_{k+1}=x(t_{k+1})$.
    \State $k\gets k+1$; $t_k\gets t_{k-1}+\Delta T$.
    \Statex\hspace{1em}\rule{\linewidth-3em}{.4pt}
\EndWhile
\end{algorithmic}
\end{algorithm}

\section{Numerical Examples}
\label{section:numerical}

This section shows how the optimal-control approach developed above is implemented in practice and integrated into a receding-horizon framework for the test scenarios. All simulations use an indirect multiple-shooting method: We start by assuming a sequence of dynamic phases—typically a constant-speed plateau, a boundary ramp limited by the BLN gradient constraint, and a second plateau—with initial durations for each phase. The state equations are integrated forward in time while the costate system is integrated backward; a nonlinear solver adjusts the phase lengths and multipliers until the matching and complementarity conditions are met. When residuals remain significant, the phase pattern is updated based on the switching function~$\Phi$, and the shooting process is repeated.

To avoid local minima and improve convergence, the indirect solver uses the previous MPC solution as a starting point, shifted by one time step. 

The predictive controller updates every 60 seconds. Short-term arrival-rate forecasts from a first-order autoregressive filter drive the open-loop optimization over a prediction horizon of 4 hours. Forecast errors are naturally corrected at the next sampling step through measured states. 

The model parameters are based on the estimated values from Hammerl et al. \cite{Hammerl2024}, which were derived from the analysis of a section of Interstate I-880 N. A detailed description of the examined road section can be found in  \cite{Munoz2002}.
Traffic flow occurs along a homogeneous 10-kilometer corridor and follows a triangular fundamental diagram. The characteristic parameters include a free-flow speed of 120 km/h, a congestion wave speed of 24 km/h, and a maximum capacity of 8,400 vehicles per hour. Permissible speed limits range from 40 to 120 km/h. The time-dependent traffic demand is modeled using the following piecewise linear function: \[
d(t) =
\begin{cases}
1714.08\, t + 5571.84, & 0 \leq t \leq 2 \\
-1892.335\, t + 12784.67, & 2 < t \leq 4 \\
0, & t > 4
\end{cases}
\] The peak demand was increased compared to the original estimate to better demonstrate congestion caused by VSL adjustments. 

Initially, three model instances are simulated: a reference scenario ($\mu_1$ = 1) and two variants with moderate ($\mu_1$ = 2/3) and high priority ($\mu_2$ = 1/3) for the emission target, respectively. The emissions of individual pollutants are equally weighted ($\lambda_1 = \lambda_2 = \lambda_3 = 1/3$). The results are presented in Figures \ref{figure:vsl_tt_basic} and \ref{figure:em_cost_basic} and in Table \ref{table:sum_basic}.

\begin{figure}[H]
  \centering
  \begin{subfigure}[b]{0.9\textwidth}
    \centering
    \includegraphics[width=\textwidth]{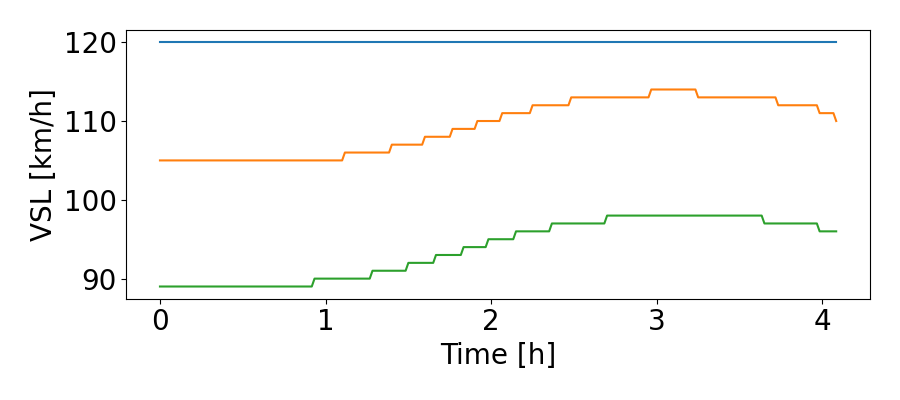}
    \caption{VSL trajectories}
  \end{subfigure}
  
  \vspace{1em} 
  
  \begin{subfigure}[b]{0.9\textwidth}
    \centering
    \includegraphics[width=\textwidth]{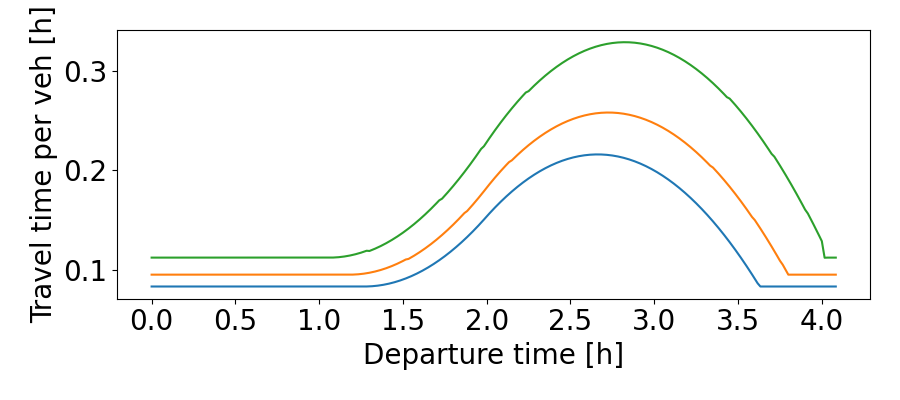}
    \caption{Travel times}
  \end{subfigure}
  
  \caption{VSL trajectories and travel times for all three scenarios: baseline (blue), moderate emission sensitivity (orange), and high emission sensitivity (green).}
  \label{figure:vsl_tt_basic}
\end{figure}

\begin{figure}[H]
  \centering
  \begin{subfigure}[b]{0.9\textwidth}
    \centering
    \includegraphics[width=\textwidth]{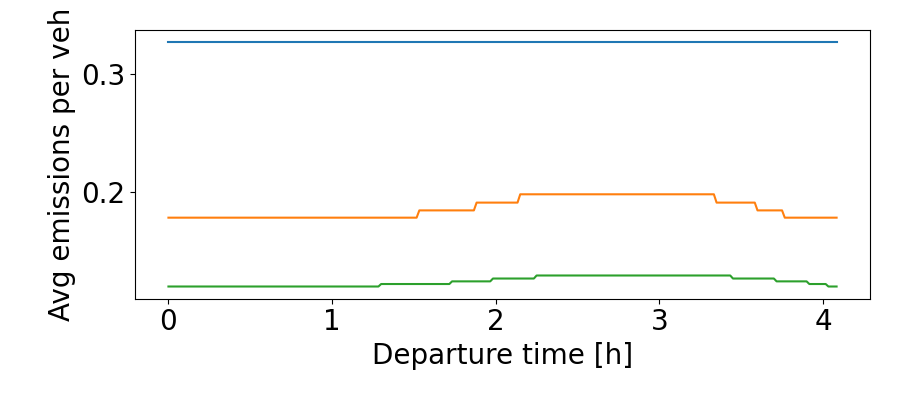}
  \end{subfigure}
  
  \vspace{1em}
  
  \begin{subfigure}[b]{0.9\textwidth}
    \centering
    \includegraphics[width=\textwidth]{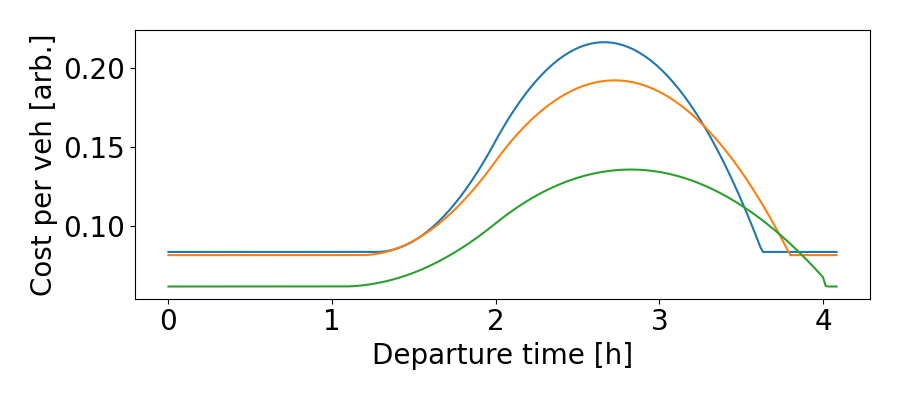}
  \end{subfigure}
  
  \caption{Emissions and combined running cost per vehicle for all three scenarios.}
  \label{figure:em_cost_basic}
\end{figure}

\begin{table}[H]
\centering
\begin{tabular}{lccc}
\toprule
\textbf{Scenario} &
\makecell{\textbf{Avg Travel}\\\textbf{Time [min/veh]}} &
\makecell{\textbf{Avg Emissions}\\\textbf{[g/veh]}} &
\makecell{\textbf{Avg Objective}\\\textbf{[arb./veh]}} \\
\midrule
TT only      & 7.89006  & 0.327090 & 0.131501 \\
33\% Emiss.  & 9.56334  & 0.187344 & 0.124994 \\
67\% Emiss.  & 12.39678 & 0.124697 & 0.093811 \\
\bottomrule
\end{tabular}
\caption{Comparison of travel time, emissions, and objective values for different VSL strategies.}
\label{table:sum_basic}
\end{table}

In the reference case, the form of the optimal trajectory is evident: both under free-flow conditions and during congestion formation, maximum throughput and free-flow speed are achieved by maintaining a constant speed limit of 120 km/h. For low and high emission priority, respectively, the set speed limit begins at values of approximately 105 and 90 km/h, which generate minimum costs for the respective weighting of the linear combination of travel time and vehicle emissions.
During congestion formation upstream of the VSL location, the actual travel time falls below the optimal travel time. The VSL algorithm compensates for this by raising the speed limit downstream of the VSL location (to approximately 105 and 115 km/h, respectively). The marginal effect of reduced free-flow travel time outweighs the capacity reduction induced by this increase at the VSL location.
A mathematically obvious but sometimes overlooked fact is that drivers typically have an approximately constant value of travel time, and the marginal benefit of speed increases diminishes at high speeds. Conversely, emitted pollutants are represented by a convex function of average speed, so that at high speeds, a speed reduction can achieve significant emission savings with only minimal travel time increases. 

Finally, the robustness of the control scheme against demand fluctuations is examined. For both emission priorities (33\% and 67\%), the deviations of actual from optimal costs specified in Table~\ref{table:sum_basic} are calculated using Monte Carlo simulation with 10{,}000 runs. The upstream demand is modeled as a temporally fluctuating random variable \(\tilde{d}(t) = d(t) + \varepsilon(t)\), calculated every second, where \(\varepsilon(t) \sim \mathcal{N}(0, 0.02^2 f^2(t))\). Two simulation series are conducted: one with uncorrelated fluctuations and a second with a correlation of 0.8 between consecutive values. The results for relative regret \((J_{\text{noise}} - J_{\text{base}})/J_{\text{base}}\) are presented by quantiles in Figures~\ref{figure:regret_33} and~\ref{figure:regret_67}, and for selected percentiles in Table~\ref{table:regret_summary}. The maximum deviation from the baseline value amounts to only approximately one standard deviation of the fluctuation term, even in unlikely extreme cases. For the lowest quantiles, the relative regret is even negative, so that favorable deviations - typically lower actual traffic volumes compared to forecasts - can be efficiently exploited by the algorithm.

\begin{figure}[H]
    \centering
    \includegraphics[width=0.8\textwidth]{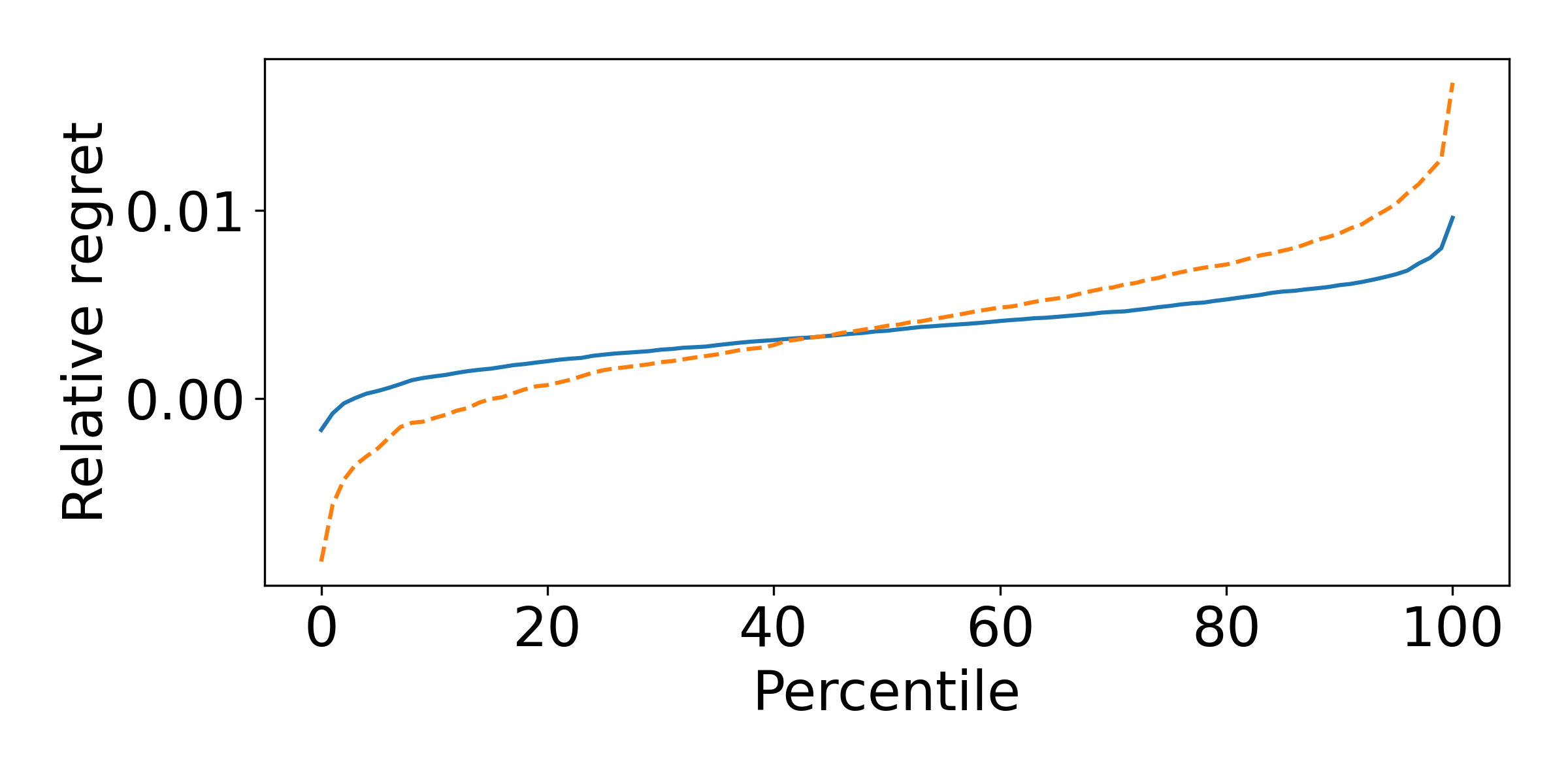}
    \caption{Percentiles of the relative regret for emission priority \(\sigma_2 = \frac{1}{3}\); uncorrelated demand noise shown in blue, correlated noise (\(\rho = 0.8\)) in orange.}
    \label{figure:regret_33}
\end{figure}

\begin{figure}[H]
    \centering
    \includegraphics[width=0.8\textwidth]{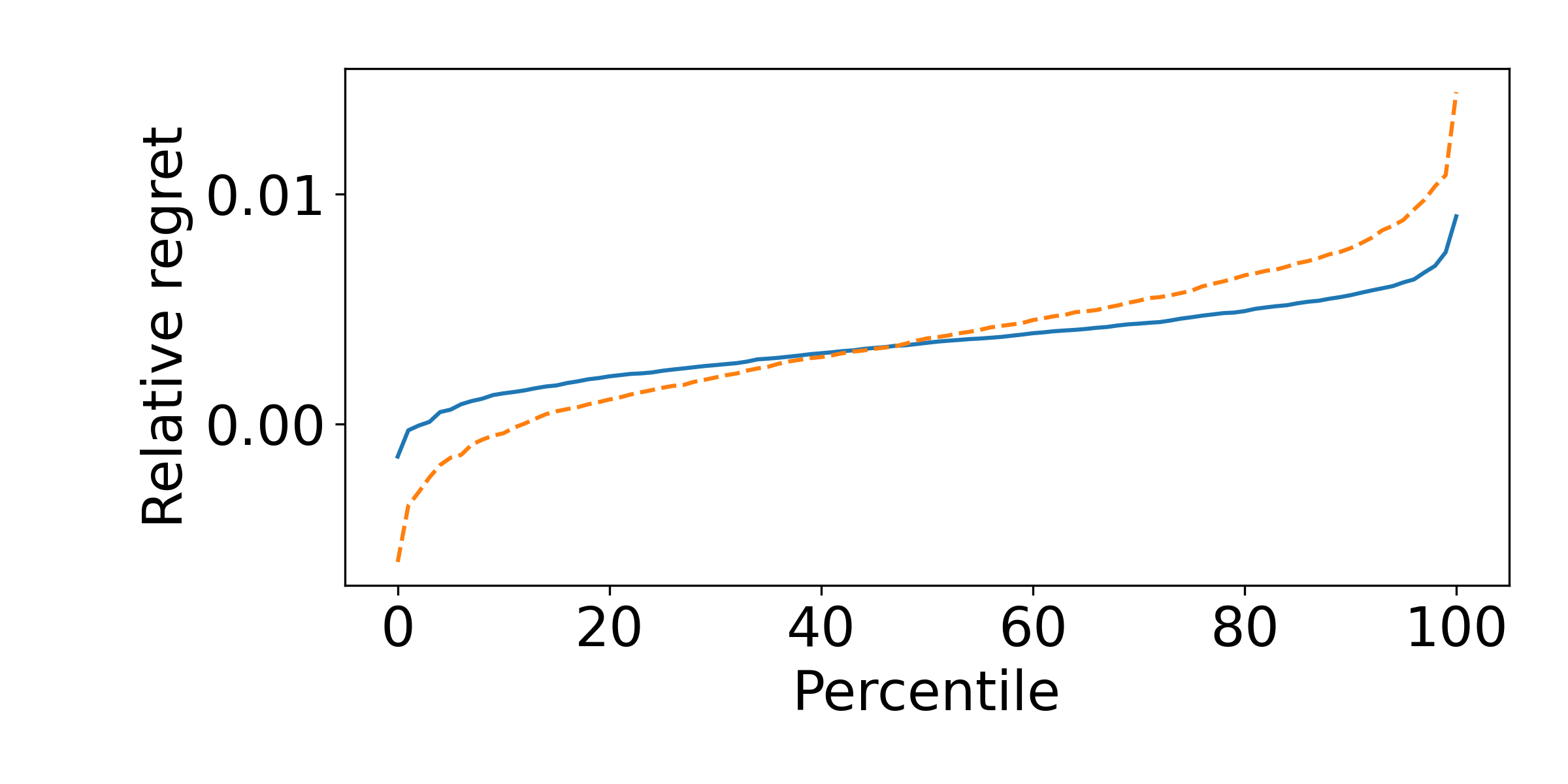}
    \caption{Percentiles of the relative regret for emission priority \(\sigma_2 = \frac{2}{3}\).}
    \label{figure:regret_67}
\end{figure}

\begin{table}[H]
\centering
\begin{tabular}{llcccc}
\toprule
 & Scenario & 50\% & 80\% & 95\% & Worst \\
\midrule
0 & 33\% (white) & 0.003602 & 0.005268 & 0.006612 & 0.009617 \\
1 & 33\% (corr)  & 0.003866 & 0.007129 & 0.010354 & 0.016807 \\
2 & 67\% (white) & 0.003525 & 0.004900 & 0.006148 & 0.009020 \\
3 & 67\% (corr)  & 0.003716 & 0.006457 & 0.008862 & 0.014415 \\
\bottomrule
\end{tabular}
\caption{Selected quantiles and worst-case values of the relative regret for different emission priorities and noise correlation structures.}
\label{table:regret_summary}
\end{table}

We assess the robustness of the MPC with a two–factor analysis of variance (ANOVA) \cite{Montgomery2017}. The factors are (i) the objective weighting, $\mu_{\mathrm{emiss}}\!\in\!\{0.33,0.67\}$, and (ii) the disturbance model, \textit{white} versus \textit{AR(1)} demand noise ($\rho = 0.8$), while every Monte–Carlo seed provides a replicate. ANOVA decomposes the total scatter of the relative–regret metric of the 1\,000 runs into additive contributions from the factors and their interaction:

\begin{table}[htbp]
\centering
\small
\begin{tabular}{lccc}
\toprule
\textbf{Source} & \textbf{\% Variance} & \textbf{$F$} & \textbf{$p$--value} \\
\midrule
$\mu$--setting (33\,\% vs 67\,\%)                       & 0.02\,\%   & 0.77 & 0.38 \\
Noise type (white vs AR(1), $\rho = 0.8$)               & 0.21\,\%   & 8.51 & 0.004 \\
Interaction $\mu \times$ noise                          & 0.003\,\%  & 0.14 & 0.71 \\
Residual (run--to--run randomness)                      & 99.77\,\%  & \multicolumn{2}{c}{---} \\
\bottomrule
\end{tabular}
\caption{ANOVA decomposition of the relative–regret metric over 1\,000 Monte–Carlo runs.}
\label{tab:anova}
\normalsize
\end{table}

The \emph{noise type} is formally significant ($p<0.01$) but explains only ${\approx}\,0.2\%$ of the total variance in regret, while the choice of emission weight and the interaction term are both statistically and practically negligible.  
Consequently, {${\sim}\,99.8\%$} of the variability stems from the $\pm2\%$ second-by-second demand perturbations themselves, not from the controller’s tuning or from the correlation structure of the disturbance. Hence the MPC’s performance is essentially invariant to these structured uncertainties, confirming its robustness against realistic demand fluctuations.

\section{Conclusion}
\label{section:conclusion}
This paper has presented a novel Model Predictive Control framework for reducing traffic emissions on suburban freeways through the implementation of Variable Speed Limits. Our approach leverages a bounded-acceleration extension of the LWR traffic flow theory to achieve computationally efficient control while maintaining behavioral realism in congestion dynamics.
Our primary contribution demonstrates that by applying a BA extension of LWR theory, vehicle travel times essential for control applications can be calculated analytically and efficiently from kinematic principles. While analytical works on state-of-the-art higher-order macroscopic models, which derive their foundations from fluid mechanics and gas particle movement, typically focus on derivations of flow and density relationships, vehicle travel times represent the practically more critical parameter for practical control applications. Through our first-order bounded acceleration extension, by extending the methodology established in \cite{Lebacque2003} and \cite{Jin2018}, we successfully derive convenient analytical solutions while qualitatively capturing all essential phenomena of congestion formation and dissipation at bottlenecks.
The bounded acceleration framework provides several key advantages over traditional LWR models. By incorporating finite acceleration and reaction time constraints that reflect realistic vehicle capabilities, our model addresses the well-known deficiency of infinite accelerations inherent in classical kinematic wave theory. This enhancement enables more accurate prediction of capacity recovery processes and provides a more realistic representation of driver behavior during congestion transitions. Furthermore, our analytical treatment of the acceleration phase establishes closed-form expressions for flow evolution at Variable Speed Limit locations, enabling real-time optimization without requiring computationally intensive numerical integration schemes.
Our integration of the COPERT V emission model within the MPC framework represents a significant advancement in emission-aware traffic control. By incorporating speed-dependent emission factors for multiple pollutants (CO, NOx, and HC) directly into the optimization objective, we demonstrate that substantial emission reductions can be achieved with only modest increases in travel times, highlighting the significant potential for environmental benefits through intelligent speed management.
The robustness analysis conducted through Monte Carlo simulations with demand uncertainty further validates the practical applicability of our approach. Even under significant demand fluctuations with temporal correlation, the relative regret remains bounded within acceptable limits, demonstrating the controller's ability to adapt to real-world traffic variability. This robustness is particularly important for suburban freeway applications where demand patterns can be highly variable and difficult to predict accurately.

Several promising research directions emerge from this work that warrant further investigation. From a traffic flow modelling standpoint, it could be insightful to examine the effects of applying state-of-the-art higher-order traffic flow models to this control setting and comparing outcomes with our modeling approach. Higher-order models typically explain dynamic changes in traffic conditions through a duality of driver reaction and anticipation mechanisms, whereas the boundary conditions employed in our BA LWR model correspond primarily to reaction terms. This fundamental difference has important implications for control performance that merit careful study. First-order kinematic wave models, which incorporate neither reaction nor anticipation terms, face the inherent risk of inducing the exact capacity drop phenomenon at VSL locations that they seek to prevent at bottleneck locations. This occurs due to their inability to account for the endogenous nature of capacity drops. Additionally, these models may overestimate the efficacy of control speed increases by overestimating the propagation speed of signal changes. Conversely, bounded acceleration models like ours, which incorporate only reaction terms(see also e.g. \cite{Leclercq2007}, \cite{Piacentini2018},\cite{Wang2022}) for other models and applications), may penalize control signal changes too severely by assuming that any capacity drop incurred by such changes is permanent. These models potentially underestimate drivers' ability to recover initial equilibrium states through anticipation of traffic conditions ahead.
A systematic comparison between our BA-LWR approach and second-order models incorporating both reaction and anticipation terms would provide valuable insights into the trade-offs between model complexity, computational efficiency, and control performance. Such studies could inform the development of hybrid approaches that capture the most beneficial aspects of each modeling paradigm while maintaining computational tractability for real-time applications.
Another significant research opportunity lies in extending this control approach to non-spatial or hybrid representations of entire transportation networks (see e.g. \cite{Geroliminis2008},\cite{Daganzo2008}, \cite{Jin2020}, \cite{Haddad2013}). Such extensions would enable a holistic approach to vehicle pollution management in urban environments, moving beyond freeway corridors to consider network-wide interactions and system-level optimization. This expansion could address important questions about the spatial and temporal coordination of multiple VSL and metering systems and their cumulative impact on urban air quality.
Additionally, future work should explore the integration of emerging technologies such as connected and autonomous vehicles into the control framework. These technologies offer the potential for more precise implementation of speed recommendations and could enable anticipatory control strategies that leverage vehicle-to-infrastructure communication. The development of adaptive emission models that account for evolving vehicle fleet compositions and engine technologies would also enhance the long-term applicability of emission-based traffic control strategies.

\end{document}